\documentclass[10pt, reqno]{amsart}

\usepackage{amssymb}
\usepackage{amsmath}
\usepackage{amsthm}
\usepackage{mathrsfs}
\usepackage{dsfont}
\usepackage{url}

\newtheorem{theorem}{Theorem}
\newtheorem{lemma}[theorem]{Lemma}
\newtheorem{conjecture}[theorem]{Conjecture}

\numberwithin{equation}{section}
\numberwithin{theorem}{section}

\begin{document}

\title[Jumping champions and gaps between consecutive primes]{Jumping champions and gaps between consecutive primes}

\author[D. A. Goldston]{D. A. Goldston$^{*}$}

\address{Department of Mathematics, San Jos\'{e} State University, One Washington Square\\
 San Jos\'{e}, California 95192-0103, United States of America}

\thanks{$^{*}$Research of the first author was supported in part by the National Science Foundation Grant DMS-0804181.}

\email{goldston@math.sjsu.edu}

\author[A. H. Ledoan]{A. H. Ledoan}

\address{Department of Mathematics, University of Rochester, Hylan Building\\
Rochester, New York 14627-0138, United States of America}

\email{ledoan@math.rochester.edu}



\begin{abstract}
The most common difference that occurs among the consecutive primes less than or equal to $x$ is called a jumping champion.  Occasionally there are ties. Therefore there can be more than one jumping champion for a given $x$.  In 1999 A. Odlyzko, M. Rubinstein, and M. Wolf provided heuristic and empirical evidence in support of the conjecture that the numbers greater than 1 that are jumping champions are  4 and the primorials 2, 6, 30, 210, 2310, \ldots.  As a step towards proving this conjecture they introduced a second weaker conjecture that any fixed prime $p$ divides all sufficiently large jumping champions. In this paper we extend a method of P. Erd\H{o}s and E. G. Straus from 1980 to prove that the second conjecture follows directly from the prime pair conjecture of G. H. Hardy and J. E. Littlewood.
\end{abstract}

\maketitle

\thispagestyle{empty}

\section{The most likely differences between consecutive primes}	

In an issue of the 1977-1978 volume of Journal of Recreational Mathematics, H. Nelson \cite{Nelson1977-1978} proposed the following unsolved Problem 654:
\begin{quote}
``Find the most probable difference between consecutive primes.''
\end{quote}
The following year the 1978-1979 volume of the journal had  the Editor's  Comment \cite{Nelson1978-1979}:
\begin{quotation}
``No solution has been received, though there has been a good deal of evidence presented pointing to the reasonable conjecture that there is no most probable difference between consecutive primes. On the other hand, there is also some evidence that 6 is the most probable such difference, and it is known from computer counts that 6 is the most probable difference for primes less than $10^9$. However, there seems to be good reason to expect that 30 will eventually replace 6 as the most probable difference and still later 210, 2310, 30030, etc. will have their day.''
\end{quotation}
As indicated subsequently in the same comment, Nelson was motivated to ask his question by a statement in Popular Computing Magazine that 6 appears to be the most common distance between primes. The editor concluded that: ``The problem will be left open for another year.''

In 1980 P. Erd\H{o}s and E. G. Straus \cite{ErdosStraus1980} showed, on the assumption of the truth of the prime pair conjecture of G. H. Hardy and J. E. Littlewood \cite{HardyLittlewood1923}, that there is no most likely difference because the most likely difference grows as one considers larger and larger numbers. In 1993 J. H. Conway invented the term jumping champion to refer to the most frequently occurring difference between consecutive primes less than or equal to $x$. For the $n$th prime $p_n$, the jumping champions are the values of $d$ for which
\[
N(x,d)\mathrel{\mathop:}= \sum_{\substack{p_{n}\leq x \\ p_{n} - p_{n - 1} = d}} 1
\]
attains its maximum
\[
N^{*}(x)
 \mathrel{\mathop:}= \max_{d} N(x,d).
\]
Thus the set $D^{*}(x)$ of jumping champions for primes less than or equal to $x$ is given by
\[
D^{*}(x)
 \mathrel{\mathop:}= \{ d^{*} \colon N(x,d^{*}) = N^{*}(x)\}.
\]

Table \ref{table1} below summarizes  everything we presently know about jumping champions less than or equal to $x$.  Professor M. Wolf has kindly informed us that the search for the first occurrences of prime gaps and maximal prime gaps was extended in 1999 to $10^{15}$ by T. R. Nicely \cite{Nicely1999} and that 6 is the jumping champion. Other than for the sets $D^{*}(3) = \{1\} $ and $D^{*}(5) = \{1,2\}$, it is not known if the numbers indicated in the right-most column of the table are the largest prime occurrence of $x$ for the corresponding set $D^{*}(x)$ in the left-most column.

\begin{table}[ht] \label{table1}
\caption{Known jumping champions for small $x$.}
{\begin{tabular}{@{}ccc@{}}
$D^{*}(x)$		&Smallest Prime	&	Largest Known Prime		\\
			&Occurrence of $x$	&	Occurrence of $x$			\\
\hline
$\{1\}	$		&	3			&	3						\\
$\{1, 2\}$		&	5			&	5						\\
$\{2\}	$		&	7			&	433						\\
$\{2, 4\}$		&	101			&	173						\\
$\{4\}$		&	131			&	541						\\
$\{2, 4, 6\}$	&	179			&	487						\\
$\{2, 6\}$		&	379			&	463						\\
$\{6\}	$		&	389			&	$>10^{15}$				\\
$\{4, 6\}$		&	547			&	941						\\
\end{tabular}}
\end{table}

In their paper \cite{OdlyzkoRubinsteinWolf1999} of 1999 A. Odlyzko, M. Rubinstein, and M. Wolf made the following two simple and elegant conjectures.

\begin{conjecture}\label{conjecture1}
The jumping champions greater than 1 are 4 and the primorials 2, 6, 30, 210, 2310, \ldots.
\end{conjecture}

\begin{conjecture}\label{conjecture2}
The jumping champions tend to infinity. Furthermore, any fixed prime $p$ divides all sufficiently large jumping champions.
\end{conjecture}

Conjecture \ref{conjecture2} is an immediate consequence of Conjecture \ref{conjecture1}. Odlyzko, Rubinstein, and Wolf remarked that Conjecture \ref{conjecture2} ``should be considerably easier to prove, and might conceivably be provable unconditionally.'' As already mentioned, the first assertion of Conjecture \ref{conjecture2} was  proved in 1980 by Erd\H{o}s and Straus \cite{ErdosStraus1980}, under the assumption of the truth of the Hardy-Littlewood prime pair conjecture. Here we extend the argument of Erd\H{o}s and Straus to give a complete proof of Conjecture \ref{conjecture2}, again under the assumption of the Hardy-Littlewood prime pair conjecture.

\section{The Hardy-Littlewood prime pair conjecture}\label{section2}

The Hardy-Littlewood prime pair conjecture is that
\begin{equation} \label{eq1}
\pi_2(x,d)
 \mathrel{\mathop:}= \sum_{\substack{p \leq x \\ p - p' = d}} 1
 \sim\mathfrak{S}(d) \frac{x}{(\log x)^2}, \quad \mbox{as $x \to \infty$},
\end{equation}
for even integers $d\ge 2$, where the sum is taken over primes $p \leq x$ such that $p - p' = d$. Here $p'$ is a previous prime before $p$ but not necessarily adjacent to $p$. Moreover, the singular series $\mathfrak{S}(d)$ is defined for all integers $d \neq 0$ by
\[
 \mathfrak{S}(d) = \left\{ \begin{array}{ll}
      {\displaystyle 2C_2\prod_{\substack{p \mid d \\ p > 2}} \left(\frac{p - 1}{p - 2}\right),} & \mbox{if $d$ is  even;} \\
      0,   & \mbox{if $d$ is  odd;} \\
\end{array}
\right.
\]
where
\[
C_2
 = \prod_{p > 2}\left(1 - \frac{1}{(p - 1)^2}\right)
 = 0.66016\ldots
\]
with the product extending over all primes $p > 2$.
It is reasonable to suppose that the Hardy-Littlewood prime pair conjecture will hold uniformly for any integer $d$ in the range $2\leq d \leq x$, but we will need it only in a shorter range. 

\begin{theorem} \label{main_theorem}
Assume that the Hardy-Littlewood prime pair conjecture \eqref{eq1} holds uniformly for integers $d$ in the range $2\leq d \leq (\log x)^2$ for all sufficiently large $x$. Then Conjecture \ref{conjecture2} is true. 
\end{theorem}

The property of the singular series $\mathfrak{S}(d)$ most crucial for jumping champions is that $\mathfrak{S}(d)$ increases most rapidly on the sequence of primorials. This follows readily from the formula
\begin{equation}\label{eq2}
\mathfrak{S}(d)
 = 2C_2\prod_{\substack{p \mid d \\ p > 2}}\left(1+ \frac{1}{p - 2}\right)
\end{equation}
for even $d$. It is evident that each prime divisor of $d$ contributes to increase the size of $\mathfrak{S}(d)$. The following lemma will be used throughout the paper. 

\begin{lemma}\label{auxiliary_lemma}
If
\[
\mathcal{P}_k
 \mathrel{\mathop:}= 2\cdot 3\cdot 5 \cdots p_k
\]
denotes the $k$th term in the sequence of primorials, then the inequality
\[
\mathfrak{S}(d) < \mathfrak{S}(\mathcal{P}_k)
\]
holds for every integer $d$ such that $2 \leq d < \mathcal{P}_k $.
\end{lemma}

\begin{proof}
We may suppose, since the value of $\mathfrak{S}(d)$ does not depend on the multiplicity of the prime factors, that the integer $d$ is square-free and even. We observe on the one hand  that if $d \mid \mathcal{P}_k$, then by virtue of \eqref{eq2} the result is clear. On the other hand, if $d$ has a prime factor greater than $p_k$, then since $d < \mathcal{P}_k$ we see that $d$ must be short of some prime factors less than $p_k$.  Thus we can increase the size of $\mathfrak{S}(d)$ by replacing the larger prime factor in $d$ by a smaller prime factor that $d$ lacks. This process can be continued until a new integer $d'$ is obtained such that $\mathfrak{S}(d) <  \mathfrak{S}(d')$ and $d' \mid \mathcal{P}_k$. This proves Lemma \ref{auxiliary_lemma}.
\end{proof}

\section{Proof of Theorem \ref{main_theorem}}

The Hardy-Littlewood prime pair conjecture \eqref{eq1} gives an asymptotic formula for the number of differences $p - p' = d$ for primes $p$ and $p'$ with $p\leq x$. These primes need not be consecutive, but Erd\H{o}s and Straus proved that asymptotically the number of such differences that are not between consecutive primes is insignificant as long as $d$ does not grow very rapidly with $x$. For if $p$ and $p'$ are not consecutive primes, then there must exist a third prime $p''$ such that $p'<p''<p$. Writing $d'= p-p''$, we count the number of such triplets of primes with
\[
\pi_3(x,d,d')
 \mathrel{\mathop:}= \sum_{\substack{ p \leq x \\ p - p' = d \\ p - p'' = d'}} 1
\]
and obtain
\begin{equation}\label{eq3}
\pi_2(x,d) - \sum_{1 \leq d' < d} \pi_3(x,d,d')
 \leq N(x,d)
 \leq \pi_2(x,d).
\end{equation}

An application of either the Brun or Selberg sieve (see H. Halberstam and H.-E. Richert's classical monograph \cite{HalberstamRichert1974}, Chapters 2 through 5) then gives the upper bound
\begin{equation}\label{eq4}
\pi_3(x,d,d')
 \ll \mathfrak{S}(\{0,d',d\}) \frac{x}{(\log x)^3},
\end{equation}
where $\mathfrak{S}(\{0,d',d\}) $ denotes the singular series associated with this triple. We will employ the elementary bound
\begin{equation}\label{eq5}
\mathfrak{S}(\{0,d',d\})
 \ll d^{\varepsilon}, \quad \mbox{for all $\varepsilon > 0$,}
\end{equation}
for which a short proof is given in Section \ref{section4}.

Now from \eqref{eq4} and \eqref{eq5} we see that
\[
\sum_{1\leq d' < d} \pi_3(x,d,d')
 \ll d^{1 + \varepsilon} \frac{x}{(\log x)^3}.
\]
Thus by \eqref{eq3} and the Hardy-Littlewood prime pair conjecture \eqref{eq1} in the range $2\leq d \leq (\log x)^2$ we have
\begin{equation}\label{eq6}
N(x,d)
 =\mathfrak{S}(d) \frac{x}{(\log x)^2}\left(1 + o(1)\right), \quad \mbox{for $2 \leq d \leq \sqrt{\log x}$,}
\end{equation}
and the upper bound
\begin{equation}\label{eq7}
N(x,d)
 \leq\mathfrak{S}(d) \frac{x}{(\log x)^2}\left(1 + o(1)\right), \quad \mbox{for $2 \leq d \leq (\log x)^2$.}
\end{equation}

Before we prove Conjecture \ref{conjecture2}, a simple deduction from \eqref{eq6} and \eqref{eq7}, let us observe that \eqref{eq7} will remain true in any range where the Hardy-Littlewood prime pair conjecture \eqref{eq1} holds true, whereas \eqref{eq6} will no longer be true if $d \gg \log x$ because at the average spacing for primes too many of the differences will no longer be between consecutive primes.

In the following, it will be convenient to define a floor function with respect to a given increasing sequence $\{a_k\}_{k = 1}^{\infty}$ by
\[
\lfloor y \rfloor_{a_k}
  \mathrel{\mathop:}= a_n, \quad \mbox{if $a_n\leq y < a_{n+1}$.}
\]
Hence $\lfloor y \rfloor_{\mathcal{P}_k}$ is the largest primorial less than or equal to $y$.

We have
\begin{equation}\label{eq8}
\mathfrak{S}(\lfloor \sqrt{\log x} \rfloor_{\mathcal{P}_k}) \frac{x}{(\log x)^2} (1 - o(1))
 \leq \max_{2 \leq d \leq \sqrt{\log x}} N(x,d)
 \leq N^{*}(x),
\end{equation}
by \eqref{eq6}. Since
\[
\begin{split} N(x,H)
 &\leq \sum_{\substack{p_n \leq x \\ p_n - p_{n - 1} \geq H}} 1
 \leq \sum_{\substack{p_n \leq x \\ p_n - p_{n - 1} \geq H}} \frac{(p_n - p_{n - 1})}{H} \\
 &\leq \frac1{H}\sum_{p_n \leq x} (p_n - p_{n - 1})
 \leq \frac{x}{H},
\end{split}
\]
we have 
\[
N(x,d)
 \leq \frac{x}{(\log x)^2}, \quad \mbox{if $d \geq (\log x)^2$.}
\]
But from \eqref{eq8} and Lemma \ref{auxiliary_lemma}
\[
N^{*}(x)
 \geq \mathfrak{S}(2) \frac{x}{(\log x)^2} (1 - o(1))
 > 1.32 \frac{x}{(\log x)^2},
\]
and so
\[
N^{*}(x)
 =\max_{2 \leq d \leq (\log x)^2} N(x,d).
\]

Suppose now that $d^{*} \in D^{*}(x)$ is a jumping champion. Then
\begin{equation}\label{eq9}
d^{*}
 \leq (\log x)^2.
\end{equation}
By \eqref{eq7} and \eqref{eq8}
\[
\mathfrak{S}(\lfloor \sqrt{\log x} \rfloor_{\mathcal{P}_k}) \frac{x}{(\log x)^2} (1 - o(1))
 \leq N(x,d^{*})
 \leq \mathfrak{S}(d^{*}) \frac{x}{(\log x)^2} (1 + o(1)),
\]
and it follows that
\[
1 - o(1)
 \leq \frac{\mathfrak{S}(d^{*})}{\mathfrak{S}(\lfloor \sqrt{\log x} \rfloor_{\mathcal{P}_k}) }.
\]
Suppose further that $p^{*}$ is a given odd prime such that $p^{*} \nmid d^{*}$. From \eqref{eq2}, \eqref{eq9}, and Lemma \ref{auxiliary_lemma}
\[
\mathfrak{S}(d^{*})
 \leq \mathfrak{S}(\lfloor (\log x)^2 \rfloor_{\mathcal{P}_k}) \left(1 + \frac{1}{p^{*} - 2}\right)^{-1}.
\]
Therefore
\begin{equation}\label{eq10}
\left(1 + \frac{1}{p^{*} - 2}\right) (1 - o(1))
 \leq \frac{\mathfrak{S}(\lfloor (\log x)^2 \rfloor_{\mathcal{P}_k})}{\mathfrak{S}(\lfloor \sqrt{\log x} \rfloor_{\mathcal{P}_k})},
\end{equation}
and it remains to consider the quantity on the right-hand side.

Here we note in particular that $\lfloor y\rfloor_{\mathcal{P}_k} = \mathcal{P}_n$ is equivalent to the inequalities
\[
\prod_{p \leq p_n} p
 \leq y
 < \prod_{p \leq p_{n + 1}} p.
\]
These, in turn, are equivalent to 
\[
\vartheta(p_n)
 \leq \log y
 < \vartheta(p_{n + 1}),
\]
where
\[
\vartheta(x)
 = \sum_{p \leq x} \log p
 \]
is Chebyshev's $\vartheta$-function, and $p$ runs over all primes less than or equal to $x$.

The relation $\vartheta(x) \sim x$ as $x \to \infty$ is equivalent to the Prime Number Theorem (see A. E. Ingham's classical tract \cite{Ingham1932}, Theorem 3, Formula (6), p. 13; and H. L. Montgomery and R. C. Vaughan's monograph \cite{MontgomeryVaughan2007}, Corollary 2.5, p. 49), from which we get that $p_n \sim \log y$ as $n \to \infty$. (Alternatively, we could have avoided the Prime Number Theorem and used Chebyshev's result that $ax \leq \vartheta(x) \leq Ax$ for all sufficiently large $x$, where $a$ and $A$ are positive constants.) From this it follows that
\begin{equation}\label{eq11}
\frac{\mathfrak{S}(\lfloor (\log x)^2 \rfloor_{\mathcal{P}_k})}{\mathfrak{S}(\lfloor \sqrt{\log x} \rfloor_{\mathcal{P}_k})}
 \leq \prod_{\frac{1}{2} (1 - o(1)) \log\log x \leq p \leq 2 (1 + o(1)) \log \log x} \left(1 + \frac{1}{p - 2}\right).
\end{equation}
To evaluate this last expression we use the asymptotic formula, due to F. Mertens (see Ingham's tract \cite{Ingham1932}, Theorem 7, Formula (23), p. 22; and Montgomery and Vaughan's monograph \cite{MontgomeryVaughan2007}, Theorem 2.7, Formula (d), p. 50),
\[
\sum_{p \leq x} \frac{1}{p}
 =\log\log x + B + O\left(\frac1{\log x}\right), \quad \mbox{as $x \to \infty$,}
\]
where $B$ is a constant. Thus the right-hand side of \eqref{eq11} is
\[
\begin{split} 
 &\leq \exp \left(\sum_{\frac{1}{3} \log\log x \leq p \leq 3\log \log x} \log \left(1 + \frac{1}{p - 2}\right)\right) \\
 &\leq \exp \left(\sum_{\frac{1}{3} \log\log x \leq p \leq 3 \log \log x} \frac{1}{p} + O\left(\sum_{n \geq \frac{1}{3} \log\log x} \frac{1}{n^2}\right)\right) \\
 &\leq \exp \left(\log\left(\frac{\log\log\log x + \log 3}{\log\log\log x - \log 3}\right) + O\left(\frac{1}{\log\log\log x}\right)\right) \\
 &\leq \exp \left(\log\left(1 + O\left(\frac{1}{\log\log\log x}\right)\right) + O\left(\frac{1}{\log\log\log x}\right)\right) \\
 &\leq 1 + O\left(\frac{1}{\log\log\log x}\right).
\end{split}
\]
We conclude that
\[
\frac{\mathfrak{S}(\lfloor (\log x)^2 \rfloor_{\mathcal{P}_k})}{\mathfrak{S}(\lfloor \sqrt{\log x} \rfloor_{\mathcal{P}_k})}
 \leq 1 + O\left(\frac{1}{\log\log\log x}\right),
\]
from which, by virtue of \eqref{eq10} we obtain 
\[
1 + \frac{1}{p^{*} - 2}
 \leq 1 + o(1).
\]
This is impossible, unless $p^{*} \to \infty$ as $x \to \infty$. Thus any fixed prime must divide all the jumping champions when $x$ is sufficiently large. This completes the proof of Theorem \ref{main_theorem}.

\section{Proof of the upper bound \eqref{eq5}}\label{section4}

We take up now the postponed proof of the upper bound \eqref{eq5}. Let $\mathcal{D} = \{0, d', d\}$. We have
\[
\mathfrak{S}(\mathcal{D})
 =\prod_p\left(1 - \frac{1}{p}\right)^{-3} \left(1 - \frac{\nu_{\mathcal{D}}(p)}{p}\right),
\]
where $\nu_{\mathcal{D}}(p)$ represents the number of distinct residue classes, modulo $p$, occupied by the elements of the set $\mathcal{D}$.

Now let
\[
\Delta = d' d (d - d').
\]
If $p\nmid \Delta$, then $0$, $d'$, and $d$ are distinct from each other modulo $p$. As a result, $\nu_{\mathcal{D}}(p) = 3$. For these terms we see that
\[
\prod_{p \nmid \Delta} \left(1 - \frac{1}{p}\right)^{-3} \left(1 - \frac{\nu_{\mathcal{D}}(p)}{p}\right)
 =\prod_{p \nmid \Delta} \left(1 - \frac{1}{p}\right)^{-3} \left(1 - \frac{3}{p}\right)
 \leq 1.
\]
We made use of the inequality $(1 - x)^{-3} (1 - 3x) \leq 1$ with $x = 1/p$. (This can be easily verified using mathematical induction to show that the inequality $1 - kx \leq (1 - x)^k$ holds for $x > 0$.) Since the remaining terms are largest when $\nu_{\mathcal{D}}(p) = 1$, we have
\[
\mathfrak{S}(\mathcal{D})
 \leq \prod_{p \mid \Delta} \left(1 - \frac{1}{p}\right)^{-2}
 \leq \left(\prod_{p \leq \Delta} \left(1 - \frac{1}{p}\right)\right)^{-2}
 \ll (\log \Delta)^2
 \ll (\log d)^2
 \ll d^{\varepsilon},
\]
using Mertens's asymptotic formula (see Ingham's tract \cite{Ingham1932}, Theorem 7, Formula (24), p. 22; and Montgomery and Vaughan's monograph \cite{MontgomeryVaughan2007}, Theorem 2.7, Formula (e), p. 50)
\[
\prod_{p \leq x} \left(1 - \frac{1}{p}\right)
 \sim \frac{e^{-\gamma}}{\log x},
\]
where $\gamma$  is Euler's constant. This proves the upper bound \eqref{eq5}.

We remark that with a little more care we could have proved here the sharp bound
\[
\mathfrak{S}(\mathcal{D})
 \ll (\log\log d)^2.
\]

\section*{Acknowledgments}

The authors would like to thank the anonymous referee for carefully reading and correcting a mistake in the original version of this paper.

\end{document}